\documentclass[a4paper,10.9pt]{amsart}

\DeclareRobustCommand{\SkipTocEntry}[5]{}

\usepackage{shadow}

\usepackage{accents}
\usepackage{marvosym}

\usepackage[T1]{fontenc}
\usepackage{empheq}
\usepackage{relsize}
\usepackage{amsmath}

\usepackage{bm}

\usepackage{upgreek}
\usepackage{ esint }
\usepackage{color}
\usepackage{comment}
\usepackage{amssymb}
\usepackage{amsfonts}
\usepackage{graphicx}

\usepackage{amscd}

\usepackage{slashed}
\usepackage{pdflscape}
\usepackage{tikz}
\usepackage{enumerate}
\usepackage{multirow}
\usepackage{stmaryrd}
\usepackage{cancel}
\usepackage{t1enc}

\usepackage{amssymb}
\usepackage{amscd}
\usepackage{scalerel,stackengine}

\usepackage{ifthen}

\usepackage{pdfsync}

\usepackage{graphicx} 
\usepackage{amsmath} 
\usepackage{amsfonts}
\usepackage{amssymb}

\usepackage{bbold}
\usepackage[linkcolor=blue]{hyperref}
\usepackage{mathrsfs}
\usepackage[colorinlistoftodos]{todonotes}
\usepackage{ytableau}

\definecolor{blue}{rgb}{.255,.41,.884} 
\definecolor{red}{rgb}{1, 0, 0} 
\definecolor{green}{rgb}{.196,.804,.196} 
\definecolor{yellow}{rgb}{1,.648,0} 
\definecolor{pink}{rgb}{1,0.5,0.5}

\newtheorem{theorem}{Theorem}[section]
\newtheorem{lemma}[theorem]{Lemma}
  
\newtheorem{proposition}[theorem]{Proposition}
\newtheorem{corollary}[theorem]{Corollary}

\theoremstyle{definition}
\newtheorem{definition}[theorem]{Definition}

\theoremstyle{remark}
\newtheorem{remark}[theorem]{Remark}

\newcommand{\II}{{\rm  I\hspace{-.2mm}I}}
\newcommand{\IIo}{\hspace{0.4mm}\mathring{\rm{ I\hspace{-.2mm} I}}{\hspace{.0mm}}}

\newcommand{\III}{{{{\bf\rm I\hspace{-.2mm} I \hspace{-.2mm} I}}{\hspace{.2mm}}}{}}
\newcommand{\IIIo}{{\mathring{{\bf\rm I\hspace{-.2mm} I \hspace{-.2mm} I}}{\hspace{.2mm}}}{}}

\newcommand{\IVo}{{\mathring{{\bf\rm I\hspace{-.2mm} V}}{\hspace{.2mm}}}{}}

\newcommand{\V}{{{\underline{\overline{\bf\rm V}}}}{}}

\newcommand{\otop}{\mathring{\top}}

\def\sideremark#1{\ifvmode\leavevmode\fi\vadjust{\vbox to0pt{\vss
 \hbox to 0pt{\hskip\hsize\hskip1em
 \vbox{\hsize2cm\tiny\raggedright\pretolerance10000
  \noindent #1\hfill}\hss}\vbox to8pt{\vfil}\vss}}}

\def\idx#1{{\em #1\/}} %

\numberwithin{equation}{section}

\newcommand{\cc}{\boldsymbol{c}}

\renewcommand{\=}{\stackrel \Sigma =}

\renewcommand\geq{\geqslant}
\renewcommand\leq{\leqslant}

\stackMath
\newcommand\reallywidehat[1]{%
\savestack{\tmpbox}{\stretchto{%
  \scaleto{%
    \scalerel*[\widthof{\ensuremath{#1}}]{\kern-.6pt\bigwedge\kern-.6pt}%
    {\rule[-\textheight/2]{1ex}{\textheight}}
  }{\textheight}%
}{0.5ex}}%
\stackon[1pt]{#1}{\tmpbox}%
}

\newcommand{\bdot }{\mathop{\lower0.33ex\hbox{\LARGE$\cdot$}}}

\definecolor{ao}{rgb}{0.0,0.0,1.0}
\definecolor{forest}{rgb}{0.0,0.3,0.0}
\definecolor{red}{rgb}{0.8, 0.0, 0.0}

\newcommand{\FFo}[1]{\mathring{\underline{\overline{\rm{#1}}}}}
\newcommand{\FF}[1]{\underline{\overline{\rm{#1}}}}

\newcommand{\ce}{\mathcal{E}}

\begin{document}


\renewcommand{\today}{}
\title{
{Higher Fundamental Forms and Warped Product Hypersurfaces
}}
%
%
%

\author{ Samuel Blitz${}^\flat$ \& Josef \v{S}ilhan${}^\sharp$}

\address{${}^\flat$
 Department of Mathematics and Statistics \\
 Masaryk University\\
 Building 08, Kotl\'a\v{r}sk\'a 2 \\
 Brno, CZ 61137} 
   \email{blitz@math.muni.cz}

\address{${}^\sharp$
 Department of Mathematics and Statistics \\
 Masaryk University\\
 Building 08, Kotl\'a\v{r}sk\'a 2 \\
 Brno, CZ 61137} 
   \email{silhan@math.muni.cz}
 
\vspace{10pt}

\begin{abstract}
Warped products are one of the simplest families of Riemannian manifolds that can have non-trivial geometries. In this article, we characterize the geometry of hypersurface embeddings arising from warped product manifolds using the language of higher (Riemannian) fundamental forms. In a similar vein, we also study the geometry of conformal manifolds with embedded hypersurfaces that admits a trivialization of the conformal metric to a product metric, with base manifold given by the embedded hypersurface. We show that the higher conformal fundamental forms play a critical role in their characterization.

\vspace{0.5cm}


\vspace{2cm}

\end{abstract}


\maketitle

\pagestyle{myheadings} \markboth{S. Blitz \& J. \v{S}ilhan}{On Fundamental Forms and Riemannian and Conformal Warped Products}


\tableofcontents

\section*{Acknowledgements}
SB is supported by the Operational Programme Research Development and Education Project No. CZ.02.01.01/00/22-010/0007541. J\v{S} is supported by the Czech Science Foundation (GACR) grant GA22-00091S.

\section*{Data Availability}
We do not analyse or generate any datasets, because our work proceeds within a theoretical and mathematical approach. One can obtain the relevant materials from the references below.

%

\section{Introduction}
The notion of a \textit{warped product} manifold was introduced by Bishop and O'Neill in 1969~\cite{bishoponeill}. Given two (pseudo-)Riemannian manifolds $(F,g_B)$ and $(B,g_B)$, a warped product manifold $F \times_f B$ is a manifold $F \times B$ equipped with a (pseudo-)Riemannian metric
\begin{equation} \label{warped}
g = f^2 g_F + g_B\,,
\end{equation}
where $f  \in C^\infty_+(B)$. Here, $B$ is called the \textit{base} manifold and $F$ is called the \textit{fiber} manifold. When the warping function $f$ is constant, the manifold is called a \textit{product manifold}. In general, warped product manifolds are emblematic of some of the simplest non-trivial geometries, and as such find broad application in general relativity. For example, both the Schwarzschild metric and the FLRW metric are warped product manifolds and are physically relevant solutions to the Einstein field equations. Indeed, partial differential equations are separable on these manifolds, and hence provide a fertile ground for solving partial differential equations in non-trivial settings.

The distinction between warped product and product manifolds is lost in the conformal geometric setting. A conformal manifold $(M,\cc)$ is a smooth manifold equipped with an equivalence class of metrics $\cc = [g]$, where equivalence between two metrics $g \sim g'$ holds when there exists some $\Omega \in C^\infty_+ (M)$ such that $g' = \Omega^2 g$. Thus, it is easy to see that if a warped product metric is an element of the conformal class, then so is a product manifold, as $f^{-2} \in C^\infty_+ (M)$.

Now, as (warped) product manifolds distinguish a family of submanifolds, it is clear that these distinguished submanifold embeddings should be have distinct geometric features, both intrinsically and extrinsically. To capture this behavior, we specifically study geometric tensor invariants of the embeddings.

The first example of such a tensor invariant was the second fundamental form, introduced by Gauss in the 19th century. In 1867, Bonnet proved the fundamental theorem of surfaces, which, in modern language, asserts that up to rigid Euclidean motions, a surface embedded in $\mathbb{R}^3$ is uniquely determined by its first and second (Riemannian) fundamental forms. This result easily generalizes to hypersurfaces embedded in $\mathbb{R}^d$. The naming of these tensors begs for explanation: is there a third fundamental form? A fourth? Indeed there are---but in the context of Euclidean embeddings, they either vanish or are easily expressible in terms of the first and second fundamental forms~\cite[Volume 3]{Spivak} (depending on the definition provided). However, recent work~\cite{BGW1} found that higher fundamental forms in a conformal setting, which capture information about the hypersurface embedding, have interesting geometric interpretations. It is the goal of this article to show that the Riemannian and conformal fundamental forms play special roles in characterizing (warped) product manifolds that arise from hypersurface embeddings in the Riemannian and conformal settings, respectively.

In Section~\ref{sec:riem}, we introduce and investigate the Riemannian fundamental forms in Riemannian product and warped product manifolds. In Section~\ref{sec:conf}, we remind the reader of conformal fundamental forms and study their behavior when a conformal hypersurface embedding admits a product manifold in its conformal class.

\subsection{Notations and Conventions}

In this article, we denote by $M$ a smooth oriented manifold of dimension $d \geq 3$ and by $\Sigma$ a $d-1$ dimensional manifold, smoothly embedded via $\iota$ into $M$. We use the overline notation $\bar{\bullet}$ to denote quantities associated with the hypersurface $\Sigma$. After equipping $M$ with a metric $g$, there exists a unique Levi-Civita connection $\nabla$ on the tangent bundle $TM$ (and its dual and tensor products). The Riemann curvature tensor of the connection is denoted by $R$ and is defined according to
$$R(x,y)z = (\nabla_x \nabla_y - \nabla_y \nabla_x - \nabla_{[x,y]}) z\,,$$
where $x,y,z \in \Gamma(TM)$ are sections of the tangent bundle and $[x,y]$ is their Lie bracket.

Throughout this manuscript, we use abstract index notation to denote sections of the tangent bundle, the cotangent bundle, and their tensor products; for example, we may write $t^a{}_{bc} \in \Gamma(TM \otimes T^* M \otimes T^* M)$, and the Einstein summation convention is used to denote contraction. We use $\wedge^k$ to denote the antisymmetric $k$th tensor power of a vector bundle and square brackets $[]$ around indices to denote antisymmetrization, $\odot^k$ to denote symmetric $k$th power of the same and round brackets $()$ around indices to denote symmetrization, and often use the subscript $\circ$ (as in $t_{(ab)_{\circ}} \in \Gamma(\odot^2_{\circ} T^* M)$) to denote the trace-free part of a section or a bundle. Furthermore, we often use subscripts to denote contraction; for example, given a vector $n \in \Gamma(TM)$ and a tensor $t_{abc} \in \Gamma(\otimes^3 T^* M)$, we denote the contraction $n^a t_{abc} =: t_{nbc} \in \Gamma(\otimes^2 T^* M)$.

Using these notations, we have that
$$x^a y^b R_{ab}{}^c{}_d z^d =x^a y^b [\nabla_a, \nabla_b] z^c\,.$$
The trace of the Riemann curvature tensor, the Ricci tensor, is denoted by $Ric_{ab} := R_{ca}{}^c{}_b$ and the scalar curvature is denoted by $Sc := g^{ab} Ric_{ab}$. The trace-free part of the Riemann curvature tensor is the Weyl tensor and is defined according to
$$W_{abcd} := R_{abcd} - g_{ac} P_{bd} + g_{ad} P_{bc} + g_{bc} P_{ad} - g_{bd} P_{ac}\,,$$
where the Schouten tensor $P$ is defined according to
$$P_{ab} := \tfrac{1}{d-2} \left(Ric_{ab} - \tfrac{1}{2(d-1)} Sc \, g_{ab} \right)\,,$$
with $J := g^{ab} P_{ab}$.
For later use, we also provide the formula for two well-known higher-order curvature quantities: the Cotton tensor and the Bach tensor, respectively:
\begin{align*}
C_{abc} :=& \nabla_a P_{bc} - \nabla_b P_{ac} \\
B_{ab} :=& \Delta P_{ab} - \nabla^c \nabla_a P_{bc} + P^{cd} W_{acbd}\,,
\end{align*}
where $\Delta :=\nabla^a \nabla_a $ is the negative Laplace operator.

\section{Riemannian Fundamental Forms and Warped Products} \label{sec:riem}
\subsection{Riemannian Fundamental Forms and Hypersurface Embeddings}
We begin by defining Riemannian fundamental forms:
\begin{definition} \label{ff-def}
Let $\iota : \Sigma \hookrightarrow (M,g)$ be a smooth hypersurface embedding and let $n \in \Gamma(T M)$ be any vector orthogonal $T\Sigma$ such that $g(n,n)|_{\Sigma} = 1$. The \textit{first fundamental form} $\bar{g} := \iota^* g$ is the pullback of $g$ to the hypersurface; this is also called the \textit{induced metric}. The \textit{second fundamental form} is defined according to $\II := \iota^* \nabla g(n,\cdot)$. We define the \textit{third fundamental form} according to $\III := \iota^* g(R(n,\cdot) n, \cdot)$, and finally, for $k \geq 4$, we define the \textit{$k$th fundamental form} using abstract index notation:
$$\FF{k}_{ab} := \iota^* (n^{a_1} \cdots n^{a_{k-3}} n^c n^d \nabla_{a_1} \cdots \nabla_{a_{k-3}} R_{cabd})\,.$$
\end{definition}
As is shown in Lemma~\ref{normal-derivs}, this definition provides a natural extension of the second fundamental form, insofar as they are related to normal derivatives of $\nabla n$. Observe that a Riemannian fundamental form $\FF{k}$ has \textit{transverse order} $k-1$, meaning that it depends on no fewer and no more than $k-1$ normal derivatives of the metric.

Generally speaking, the unique feature of the Riemannian fundamental forms is their leading behavior in terms of normal derivatives of the metric---one may always add powers of the second fundamental form at lower order without ruining this behavior. As such, there are alternative forms of the Riemannian fundamental forms that are useful in other contexts. To show leading-order equivalence of these alternative forms, we first need a special set of coordinates near a hypersurface.

In general, for any hypersurface embedding $\Sigma \hookrightarrow (M,g)$, we may express the metric (in a collar neighborhood) in the \textit{normal form}, given by 
$$g = dt^2 + \bar{g}(t,x^k)_{ij} dx^i dx^j\,,$$
where $\bar{g}(t_0,x^k)$ is the induced metric on the hypersurface defined by $t = t_0$ and coordinates $\{x^i\}$ are defined by parallel transporting coordinates on $\Sigma$ along geodesics generated by $\partial_t$. 
The 1-form $dt$ is called the \textit{unit conormal 1-form} and the dual vector field will be usually denoted by $n$; 
it provides a canonical extension of 
the unit normal vector along $\Sigma$ to $M$.
This normal form is quite useful, as it allows us to easily establish alternative formulas for the fundamental forms. Indeed, we have the following technical lemmas.
\begin{lemma} \label{normal-derivs}
Let $\iota : \Sigma \hookrightarrow (M,g)$ be a smooth embedded hypersurface with the vector field $n$ dual to 
the unit conormal 1-form. Then, for every $k \geq 2$, we have that
$$\FF{k} = \iota^* \nabla_n^{k-2} \nabla n + \text{products of lower order fundamental forms.}$$
\end{lemma}
\begin{proof}
We prove this lemma by induction. Clearly, when $k = 2$, the identity is true by explicit computation. Now, suppose that the identity holds for all $2 \leq n \leq k$. Now, we compute:
\begin{align*}
\iota^* \nabla_{n}^{k-1} \nabla_a n_b =& \iota^*  \nabla_n^{k-2} n^c \nabla_a \nabla_c n_b + \iota^* \nabla_n^{k-2} n^c R_{cabd} n^d \\
=& \iota^*\nabla_n^{k-2} \nabla_a \nabla_n n_b - \iota^*\nabla_n^{k-2} [(\nabla_a n^c)(\nabla_c n_b)] + \iota^*\nabla_n^{k-2} R_{nabn} \\
=& \iota^*\nabla_n^{k-2} R_{nabn} - \iota^*\nabla_n^{k-2} [(\nabla_a n^c)(\nabla_c n_b)] \\
=& \iota^* n^c n^d n^{a_1} \cdots n^{a_{k-2}} \nabla_{a_1} \cdots \nabla_{a_{k-2}} R_{cabd} - \iota^*\nabla_n^{k-2} [(\nabla_a n^c)(\nabla_c n_b)] \\
=& \FF{k+1}{}_{ab} -  \iota^*\nabla_n^{k-2} [(\nabla_a n^c)(\nabla_c n_b)]\,.
\end{align*}
Note that the third and fourth identities follow because, from the existence of the normal form, there exists a function $t$ such that $g(n,\cdot) = dt$, and so
$$\nabla_n n_a = n^c \nabla_c \nabla_a t = n^c \nabla_a n_c = \tfrac{1}{2} \nabla_a n^2 = 0\,.$$
Now by the Leibniz property and the induction hypothesis, it follows that $\iota^*\nabla_n^{k-2} [(\nabla_a n^c)(\nabla_c n_b)]$ can be expressed along $\Sigma$ as a sum of products of lower order fundamental forms. The lemma follows.
\end{proof}
\begin{remark}
Note that the definition of the third fundamental form given in~\cite[Volume 3]{Spivak} is $-\iota^* \nabla_n \nabla_a n_b$, and so Lemma~\ref{normal-derivs} shows that Definition~\ref{ff-def} is a reasonable extension to higher fundamental forms.
\end{remark}

Using this result, we have another straightforward lemma.  Here $\mathcal{L}$ denotes the Lie derivative.
\begin{lemma} \label{lie-derivs}
Let $\iota : \Sigma \hookrightarrow (M,g)$ be a smooth embedded hypersurface with the vector field $n$ dual to 
the unit conormal 1-form. Then, for every $k \geq 2$, we have that
$$\FF{k} = \tfrac{1}{2} \iota^* \mathcal{L}_n^{k-1} g+ \text{products of lower order fundamental forms.}$$
\end{lemma}
\begin{proof}
We prove by induction. When $k = 2$, the result holds by definition.
Next, suppose that the identity holds for all $2 \leq n \leq k$. Then,
\begin{align*}
\tfrac{1}{2} \iota^* \mathcal{L}_n^{k} g_{ab} =& \tfrac{1}{2} \iota^* \mathcal{L}_n \mathcal{L}_n^{k-1} g_{ab} \\
=& \tfrac{1}{2} \iota^* \nabla_n \mathcal{L}_n^{k-1} g_{ab} + \iota^* (\mathcal{L}^{k-1}_n g)_{c(a} (\nabla_{b)} n^c) \\
=& \tfrac{1}{2} \iota^* \nabla_n \mathcal{L}_n^{k-1} g_{ab} + \FF{k}_{c(a} \II_{b)}^c + \text{products of lower order fundamental forms} \\
=& \iota^* \nabla_n^{k-1} \nabla_a n_b + \text{products of lower order fundamental forms.}
\end{align*}
But from Lemma~\ref{normal-derivs}, the last line can be expressed in terms of $\FF{k+1}$, and so induction follows.
\end{proof}

\begin{remark}\label{weights}
Note that in all of these constructions involving equality modulo lower order terms, we may ascribe weights to better control which fundamental forms appear in each identity. Indeed, assigning a weight of $3-k$ to $\FF{k}$ and a weight of $-2$ for every contraction, it is clear that the weights of each summand in the above identities are always equal.

Additionally, observe that when the lower fundamental forms vanish, the original definition of the fundamental forms is equivalent to the constructions provided in Lemmas~\ref{normal-derivs} and~\ref{lie-derivs}.
\end{remark}

Using Lemma~\ref{lie-derivs}, for any fundamental form, one may expres
 it as a sum of Lie derivatives of the metric and their contractions:
\begin{corollary} \label{cor-lies}
Let $\iota : \Sigma \hookrightarrow (M,g)$ be a smooth embedded hypersurface with the vector field $n$ dual to 
the unit conormal 1-form. Then for every $2 \leq k$, there exists a homogeneous formula for $\FF{k}$ given by
$$\FF{k} =\tfrac{1}{2} \iota^* \mathcal{L}_n^{k-1} g +  F_k(\bar{g}^{-1}, \iota^* \mathcal{L}_n g, \ldots, \iota^* \mathcal{L}_n^{k-2} g)\,.$$
\end{corollary}
These formulas can be explicitly computed for any choice of $k$ simply by a recursive algorithm.

Lemma~\ref{lie-derivs} also implies that the metric in a sufficiently small neighborhood around a point in $\Sigma$ 
(where $(M,g)$ is analytic and $\Sigma$ is analytically embedded) is uniquely determined by the induced metric on $\Sigma$ and the infinite family of fundamental forms for the embedding:
\begin{proposition} \label{unique}
Let $\Sigma \hookrightarrow (M,g)$ be an analytic hypersurface embedding. Then, for any point $p \in \Sigma$, there exists a sufficiently small neighborhood $U_p$ around $p$ such that $g|_U$ is uniquely determined (up to diffeomorphism) by the induced metric and the infinite family of fundamental forms.
\end{proposition}
\begin{proof}
Using Corollary~\ref{cor-lies}, one may invert the formula for each $\FF{m+1}$ to find a formula for each Lie derivative of $\iota^* \mathcal{L}_n^m g$ in terms of the $\{\bar{g}^{-1}, \II, \ldots, \FF{m+1}\}$. Now, such an embedding may always be expressed in terms of a normal form (via an isometry to a collar neighborhood), so that
$$g = dt^2 + \bar{g}(t,x^k)_{ij} dx^i dx^j\,,$$
where the component functions $\bar{g}(t,x^k)_{ij}$ are analytic. But then $\bar{g}(t,x^k)$ is uniquely determined (in a sufficiently small neighborhood around a point $p \in \Sigma$) by its Taylor series expansion about $t = 0$. But because $n = \partial_t$ and $\partial_a \partial_t = 0$, it follows that $\iota^* \mathcal{L}_{n}^m g = (\partial_t^{m} \bar{g})(0,x^k)_{ij} dx^i dx^j$, and so $\bar{g}(t,x^k)$ is determined uniquely (in a sufficiently small neighborhood) by the family of tensors $\{\iota^* \mathcal{L}_n^m g\}_{m=0}^{\infty}$. However, as established, there exists a formula for each of these Lie derivatives in terms of $\bar{g}^{-1}$ and the fundamental forms. The proposition follows.
\end{proof}

This uniqueness result is quite powerful, as it allows us to compare the fundamental forms of hypersurface embeddings to determine whether those hypersurface embeddings are equivalent.
\begin{theorem} \label{equal-metrics}
Let $\iota : \Sigma \hookrightarrow (M,g)$ and $\hat{\iota} : \Sigma \hookrightarrow (M,\hat{g})$ be two analytic hypersurface embeddings such that, for every $p \in \Sigma$,  $\iota(p) = \hat{\iota}(p)$. Then, there is an analytic diffeomorphism $\phi$ defined in a collar neighborhood around $\Sigma$ that fixes $\Sigma$ pointwise such that $\phi^* \hat{g} = g$ near $\Sigma$ if and only if $\iota^* g = \hat{\iota}^* \hat{g}$ and, for every $k$, we have that $\FF{k} = \hat{\FF{k}}$.
\end{theorem}
\begin{proof}
First, suppose that a diffeomorphism $\phi$ fixing $\Sigma$ pointwise exists so that $\phi^* \hat{g} = g$. Then it follows trivially that all geometric quantities are equal, including the induced metric and the fundamental forms. On the other hand, the other direction follows directly from the uniqueness of Proposition~\ref{unique} and the existence of the normal form.
\end{proof}

As a consequence of Theorem~\ref{equal-metrics}, we can determine when a hypersurface embedding is actually a product manifold.
\begin{corollary} \label{prod-metric}
Let $\Sigma \hookrightarrow (M,g)$ be an analytic hypersurface embedding. Then $\FF{k} = 0$ for every $k \geq 2$ if and only if in a collar neighborhood $U := (-\epsilon,\epsilon) \times \Sigma$ of $\Sigma$, $(U,g|_U)$ is a product manifold.
\end{corollary}
\begin{proof}
We define a product manifold on $(-\epsilon,\epsilon) \times \Sigma$ by the metric $g = dt^2 + \iota^* g$. Clearly, this product manifold has vanishing fundamental forms and the induced metric on $\Sigma$ agrees with that of $\Sigma \hookrightarrow (M,g)$. Hence by Theorem~\ref{equal-metrics}, these metrics agree. The other direction is trivial.
\end{proof}

\subsection{Fiber-like Hypersurface Embeddings} \label{fiber}
Now, when a hypersurface $\Sigma \hookrightarrow (M,g)$ can be viewed as the fiber in a warped product, there exists a choice of local coordinates on $M$ such that the metric takes the form
$$g = dt^2 + h(t) \bar{g}_{ij} dx^i dx^j\,,$$
where $\Sigma = \mathcal{Z}(t)$ is zero locus of $t$, 
coordinates $\{x^i\}$ are coordinates on $\Sigma$ parallel transported by $\partial_t$, 
$\bar{g}_{ij}$ depends only on $x^i$,
and $h \in C^\infty_+ (\mathbb{R})$. We call such a hypersurface embedding a \textit{fiber-like hypersurface embedding}.
Note this corresponds to  \eqref{warped} with $h = f^2$ and 1-dimensional base manifold.

We can then use the results of the previous section to determine when a hypersurface embedding is indeed a fiber-like hypersurface embedding. We thus have one of the main theorems of this article.
\begin{theorem} \label{fiber-theorem}
Let $\Sigma \hookrightarrow (M,g)$ be an analytic hypersurface embedding with $\bar{g} := \iota^* g$. Then, a collar neighborhood around $\Sigma$ is isometric to a fiber-like hypersurface embedding if and only if there exists analytic $h \in C^\infty_+(\mathbb{R})$ such that for every $k \geq 2$,
$$\FF{k} = \tfrac{1}{2} h^{(k-1)}(0) \, \bar{g} + F_k(\bar{g}^{-1}, h'(0) \,\bar{g}, \ldots, h^{(k-2)}(0)\, \bar{g})\,$$
for $F_k$ from Lemma \ref{cor-lies}. That is, $\FF{k} = H\bar{g}$ where the constant $H$ is given by
$$ H = \tfrac{1}{2} h^{(k-1)}(0) + \widetilde{F}_k(1, h'(0), \ldots, h^{(k-2)}(0))  $$
where $\widetilde{F}_k$ is obtained from ${F}_k$ by replacing contraction of 2-tensor by products of corresponding constants.
\end{theorem}
\begin{proof}
First, suppose that $\Sigma \hookrightarrow (M,g)$ is an analytic fiber-like hypersurface embedding, meaning there exists some analytic $h \in C^\infty_+(\mathbb{R})$ such that, in the normal form,
$$g = dt^2 + h(t) \bar{g}_{ij} dx^i dx^j\,.$$
Clearly, because $n = \partial_t$, it follows that $\iota^* \mathcal{L}_n^{m} g = \iota^* (\partial_t^{m} h) \bar{g}$.
But from Corollary~\ref{cor-lies}, it follows that
$$\FF{k} =\tfrac{1}{2} \iota^* \mathcal{L}_n^{k-1} g +  F_k(\bar{g}^{-1}, \iota^* \mathcal{L}_n g, \ldots, \iota^* \mathcal{L}_n^{k-2} g) = \tfrac{1}{2} h^{(k-1)}(0)\, \bar{g} +  F_k(\bar{g}^{-1}, h'(0)\, \bar{g}, \ldots, h^{(k-2)}(0) \, \bar{g})\,.$$
This gives us one direction.

On the other hand, suppose that there exists an analytic $h \in C^\infty_+(\mathbb{R})$ such that for every $k \geq 2$, 
$$\FF{k} = \tfrac{1}{2} h^{(k-1)}(0) \, \bar{g} + F_k(\bar{g}^{-1}, h'(0) \,\bar{g}, \ldots, h^{(k-2)}(0)\, \bar{g})\,.$$
Now, the metric $g$ is expressible in normal form as $g = dt^2 + \bar{g}(t,x^k)_{ij} dx^i dx^j$. So,  define a metric on a collar neighborhood of $\Sigma$ by $\hat{g} = dt^2 + h(t) \bar{g}(0,x^k)_{ij} dx^i dx^j$. However, from the above considerations, this metric $\hat{g}$ induces fundamental forms that agree exactly with those of the original embedding. Hence by Theorem~\ref{equal-metrics}, $\hat{g} = g$, and so $\Sigma \hookrightarrow (M,g)$ is isometric to a fiber-like hypersurface embedding in a collar neighborhood.
\end{proof}

\begin{remark}
Explicitly, one may show that, in the language of the above theorem,
$$\III_{ab} = \tfrac{1}{2} h''(0) \bar{g}_{ab} - \tfrac{1}{4} \bar{g}^{cd} (h'(0) \bar{g}_{ac}) (h'(0) \bar{g}_{bd}) = \left[\tfrac{1}{2} h''(0) - \tfrac{1}{4} (h'(0))^2)\right] \bar{g}_{ab}\,. $$
\end{remark}

\subsection{Base-like Hypersurface Embeddings} \label{base}

Warped product manifolds with a hypersurface embedding $\Sigma \hookrightarrow (M,g)$ as the base manifold have a much richer structure. In this case, there exists some function $f \in C^\infty_+ (\Sigma)$ and some coordinates $(t,x^i)$ such that
$$g = f(x^i)^2 dt^2 + \bar{g}_{ij} dx^i dx^j\,.$$
Here $\Sigma = \mathcal{Z}(t)$ is the zero locus of $t$,
coordinates $\{x^i\}$ are coordinates on $\Sigma$ parallel transported by $\partial_t$, and $\bar{g}_{ij}$ 
depends only on $x^i$.
We call these manifolds \textit{base-like hypersurface embeddings}. Such embeddings can also be characterized in terms of fundamental forms, but their behavior is much more subtle. Note this corresponds to  \eqref{warped} with  a $(d-1)$-dimensional base manifold.

First, we prove that for any base-like hypersurface embedding, the even fundamental forms vanish.
\begin{theorem} \label{vanishing-evens}
Let $\Sigma \hookrightarrow (M,g)$ be a base-like hypersurface embedding with warping function $f$. Then, $\FF{k} = 0$ for all $k = 2n$ for positive integers $n$.
\end{theorem}
\begin{proof}
As $(M,g)$ is a warped product, in a coordinate system adapted to the warped product $(t, x^i)$, the metric is expressible as
$$g = f(x^1, \ldots, x^{d-1})^2 dt^2 + \bar{g}_{ij}(x^1, \ldots, x^{d-1}) dx^i dx^j\,.$$

Now we wish to change coordinates so that $g$ is asymptotically in normal form, i.e. so that
$$g = F^{(n+2)}(s,y^j) ds^2 + G_i^{(n+1)}(s,y^j) ds dy^i + \gamma_{ij}(s,y^1,\ldots,y^{d-1}) dy^i dy^j\,,$$
where $n \geq 0$ is even and $F^{(n+2)}(s,y^j) = 1 + \mathcal{O}(s^{n+2})$ is an even function of $s$. Similarly, we will require that $G_i^{(n+1)} = \mathcal{O}(s^{n+1})$ is an odd function of $s$.

To find such a coordinate transformation, we solve order-by-order inductively. The base case is fairly easy, with $n=0$: define $t = \frac{s}{f(y^i)}$ and $x^i = y^i$. Then, we have that
$$g = ds^2 - \frac{2s \partial_{y^i} f}{f}  dy^i ds + \gamma_{ij} dy^i dy^j \,,$$
for some $\gamma_{ij}$. Clearly, here $F^{(2)}(s,y^j) = 1$ is even and $G_i^{(1)}(s,y^j) = -\frac{2s \partial_{y^i} f}{f}$ is odd in $s$.
We now use induction to proceed. 

Suppose that $n \geq 0$ is even, that there exists $T_{n+1}(s,y^j) = \frac{s}{f(y^j)} + \cdots + Q_{n+1}(y^j) s^{n+1}$ an odd polynomial in $s$, and $X^i_{n}(s,y^j) = y^i + \cdots + P^i_{n}(y^j) s^{n}$ an even function of $s$ such that
$$g = F^{(n+2)}(s,y^j) ds^2 + G_i^{(n+1)}(s,y^j) ds dy^i + \gamma_{ij} dy^i dy^j.$$
Here we use coordinates $(s,y^j)$ related to the original coordinates $(t,x^i)$ by
$t = T_{n+1}(s,y^j)$ and $x^i = X^i_{n}(s,y^j)$, where $F^{(n+2)}$ is an even function of $s$ (with first coefficient 1)
and $G_i^{(n+1)}$ is an odd function of $s$. When $n=0$, this is the base case.

We now increment $n \rightarrow n+2$. Let $t = T_{n+1}(s,y^j) + s^{n+3} A_{n+3}(y^j)$ and $x^i = X^i_{n}(s,y^j) + s^{n+2} B^i_{n+2}(y^j)$, with $A_{n+3}$ and $B^i_{n+2}$ arbitrary functions of $\vec{y}$. We would like to show that there exist unique $A_{n+3}$ and $B^i_{n+2}$ such that 
$g_{ss} = F^{(n+4)} = 1 + \mathcal{O}(s^{n+4})$ and $g_{si} = G^{(n+3)}_i = \mathcal{O}(s^{n+3})$. 
To do so, we may compute components of the metric:
\begin{align*}
g_{si} =& G^{(n+1)}_i +  (n+2)  s^{n+1} \bar{g}_{jk} (\partial_{y^i} X_n^j) B^k_{n+2} \\
&+ s^{n+2} \left((n+3) g_{tt} A_{n+3} \partial_{y^i} T_{n+1} + \bar{g}_{jk} (\partial_s X^j_n)\partial_{y^i} B^k_{n+2} \right) \\
&+ s^{n+3} g_{tt} (\partial_s T_{n+1}) \partial_{y^i} A_{n+3} \\
&+ (n+2) s^{2n+3} \bar{g}_{jk} B^j_{n+2} \partial_{y^i} B^k_{n+2} + (n+3) s^{2n+5} g_{tt} A_{n+3}  \partial_{y^i} A_{n+3} \,.
\end{align*}
Because $\bar{g}_{jk}$ is invertible, we may fix $B^i_{n+2}$ so that the second term cancels with the leading term of $G^{(n+1)}_i$. Because $G^{(n+1)}_i$ is odd and $X^j_n$ is even (satisfying $\partial_{y^i} X^j_n = \delta_i^j + \mathcal{O}(s^2)$), the remaining higher-order terms are shunted to order $n+3$ and remain odd. Then, note that because both $T_{n+1}$ and $\partial_s X^j_n$ are odd in $s$, we see that the second line is actually odd and is of order $s^{n+3}$. Indeed, it follows that $g_{si}$ is odd and of order $s^{n+3}$. Now having fixed $B^i_{n+2}$, we may consider $g_{ss}$:
\begin{align*}
g_{ss} =& 1 + 2(n+2) s^{n+1} \bar{g}_{jk} B^j_{n+2} \partial_s X^k_n \\
&+ (F^{(n+2)} - 1)  + 2(n+3) s^{n+2} g_{tt} A_{n+3} \partial_s T_{n+1} \\
&+ (n+2)^2 s^{2n+2} \bar{g}_{jk} B^j_{n+2} B^k_{n+2} + (n+3)^2 s^{2n+4} g_{tt} A_{n+3}^2 \,.
\end{align*}
Now because $B^i_{n+2}$ is fixed and $\partial_s X^k_n$ is odd, the second term on the first line is 
even and has the factor $s^{n+2}$. The same holds for $F^{(n+2)} - 1$ and also for the second
term on the middle lide of the previous display since $T_{n+1}$ is odd. Since moreover $\partial_s T_{n+1}|_{s = 0} \neq 0$, 
we may choose $A_{n+3}$ uniquely so that the sum of first two lines of the previous display is of the form
$1 + \mathcal{O}(s^{n+4})$. 
But again, $\partial_s T_{n+1}$ is even as is $F^{(n+2)}$, so all higher-order terms remain even. Thus, the induction holds. 
Note $\gamma_{jk}$ changes in an induction step but since we did not impose any particular condition
on these functions, we keep the same notation $\gamma_{jk}$ for all induction steps.

Given such a coordinate transformation, we can now determine the parity of $\gamma_{ij}$. However, this is trivially the case: $\frac{dt}{dy^i}$ is odd, so $\frac{dt}{dy^i} \frac{dt}{dy^j}$ is even. Similarly, $\frac{dx^a}{dy^i} \frac{dx^b}{dy^j}$ is even. Therefore, we find that $\gamma_{ij}$ is an even function of $s$. Taking stock of this coordinate transform, we have that
$$g = ds^2 + (\gamma^{(0)}_{ij} + s^2 \gamma^{(2)}_{ij} + \cdots) dy^i dy^j + \mathcal{O}(s^m)$$
for any choice of $m$. Therefore, we have that in these coordinates, $n^a =  \partial_s + \mathcal{O}(s^m)$, and so $\mathcal{L}_{n} T = \partial_s T + \mathcal{O}(s^{m-1})$ for any tensor $T$. So, for any $k \leq m$, we have that $\iota^* \mathcal{L}_n^k g = \iota^* \partial_s^k g$. Notably, for any $k$ odd, we have that we may always find a coordinate transform so that
$$\iota^* \partial_s^k g = 0\,.$$

Now recall from Lemma~\ref{lie-derivs} that $\iota^* \mathcal{L}_n^k g$ is proportional to the $(k+1)$th fundamental form, modulo lower order fundamental forms. So it follows directly from  Lemma~\ref{lie-derivs} that when $k$ is odd,
$$ \FF{k+1} = \text{products of lower order fundamental forms.}$$
Here, as per Remark~\ref{weights}, each summand on the right hand side of the above display has odd weight $2-k$. But odd fundamental forms always have even weight, as does $\bar{g}^{-1}$. It thus follows that no summand on the right hand side of the above display is composed entirely of odd fundamental forms. Hence, it follows that if $\FF{m}$ vanishes for each $2 \leq m \leq k-1$ even, then it also follows that $\FF{k+1}$ vanishes. This completes the induction and hence the proof.

\end{proof}

This theorem shows that the first potentially non-vanishing fundamental form for a base-like hypersurface embedding is $\III$. It follows from the definition of $\III$ and~\cite[Proposition 42]{oneill1983} that
$$\III = f^{-1} \bar{\nabla}^2 f\,.$$
This observation is essential for the results that follow.

\begin{proposition} \label{warped-product-isometry}
Let $\Sigma \hookrightarrow (M,g)$ be a base-like hypersurface embedding. Furthermore, let $\III = 0$. Then, in a neighborhood around any point $p \in \Sigma$, $(M,g)$ is isometric to a product manifold.
\end{proposition}
\begin{proof}
If $f$ is a (nonzero) constant, the warped product is actually a product metric. If $f$ is not a constant and 
$df$ is parallel then $df$ is nonzero everywhere. Thus we can further assume  $\III = 0$, i.e.\ $\bar{\nabla}^2 f = 0$,
and $df$ is nonzero everywhere. Then both $\bar{\nabla} f$ and $\iota^* g - df \otimes df$ are parallel with respect 
to $\bar{\nabla}$. Therefore, for any point $p \in \Sigma$, there exists a neighborhood $\bar{U} \subset \Sigma$ around $p$ with coordinates $(f,y^m)$ such that the induced metric satisfies
$$\iota^* g|_{\bar{U}} = df^2 + \bar{\gamma}_{mn} dy^m dy^n\,.$$
Extending these coordinates on $\bar{U}$ to coordinates on some neighbourhood $U  \subseteq \mathbb{R} \times \bar{U}$ 
of $p$ in $M$, we obtain
$$g|_U = f^2 dt^2 + df^2 + \bar{\gamma}_{mn} dy^m dy^n\,.$$

But then the standard coordinate transformation to cartesian coordinates $f = \sqrt{u^2 + v^2}$ and $t = \arctan \tfrac{v}{u}$ yields a product metric (at least locally):
$$g = dv^2 + du^2 + \bar{\gamma}_{mn} dy^m dy^n\,,$$
and here $\Sigma = \mathcal{Z}(v)$ is the zero locus of $v$. The proposition follows.
\end{proof}

In fact, we have indentified $(M,g)$ from the previous proposition as a Riemannian product 
with a 2-dimensional Euclidean factor. But of course, it can be also considered as a product metric with a 1-dimensional
factor (as the original warped product metric).

But from this proposition and from Corollary~\ref{prod-metric}, it follows that each fundamental form vanishes, which suggests that the odd fundamental forms must be expressible as homogeneous operators on $\III$. Indeed, we have the following proposition:
\begin{proposition} \label{diff-op-FFs}
Let $\Sigma \hookrightarrow (M,g)$ be a base-like hypersurface. Then, there exists an infinite family $\{O^f_n\}_{n=1}^{\infty}$ of differential operators on $\Gamma(\odot^2 T^* \Sigma)$ depending on $f$ such that
$$\FF{2n+3} = O^f_n (\III)\,.$$
Furthermore, this operator satisfies $O^f_n(0) = 0$.
\end{proposition}
\begin{proof}
First, note that $\FF{2n+3}$ has a local, differential formula expressible entirely in terms of $f$, $\bar{g}$, their partial derivatives, and $\bar{g}^{-1}$. Now, as such objects are tensorial, we  can instead write such objects using formulas constructed entirely from $f$, its (hypersurface) covariant derivatives, the metric $\bar{g}$ and its inverse, and the curvature tensor $\bar{R}$.

Now in the case where $f$ is constant, we trivially have from Corollary~\ref{prod-metric} and Proposition~\ref{warped-product-isometry} that $\FF{2n+3} = 0$. Therefore, any such formula for $\FF{2n+3}$ cannot contain terms that only depend on $\bar{R}$ and $f$ --- i.e. all summands in such a formula must contain one or more derivatives of $f$. Thus, we may decompose any such formula into the following pieces:
$$\FF{2n+3} = O(\bar{\nabla}^2 f) + P(\bar{\nabla} f)\,,$$
where $O$ is a differential operator depending on $f$ which is  homogeneous (with respect to rescaling $f$ by a constant)  and 
$P$ is a (potentially non-linear) homogeneous algebraic operator composed of the curvature tensor $\bar{R}$, its derivatives, the metric, its inverse, and $f$.

Now consider the case where $\III = 0$, and thus $\bar{\nabla}^2 f = 0$. Because $\FF{2n+3} = 0$ by Proposition
\ref{warped-product-isometry}, it therefore follows that
$$P(\bar{\nabla} f) = 0\,.$$
However, $\bar{\nabla} f$ does not necessarily vanish (as it is independent of $\bar{\nabla}^2 f$). But because $P$ is an algebraic operator that vanishes for \textit{any} choice of $f$ satisfying $\bar{\nabla}^2 f = 0$, we must have that $P = 0$ identically. It thus follows that $\FF{2n+3} = O(\bar{\nabla}^2 f) =: O^f_n(\III)$. As $O$ is homogeneous, it trivially follows that $O^f_n(0) = 0$.
%
%
\end{proof}

\begin{remark}
As a first example, one may show, by direct computation, that for base-like hypersurface embeddings,
$$\V_{ab} =  f^{-1} (\bar{\nabla}^c f )(-3 \bar{\nabla}_c \III_{ab} + 2 \bar{\nabla}_{(a} \III_{b)c})\,.$$
Surprisingly, the operator $O^f_1$ is quite simple.
\end{remark}

From Proposition~\ref{diff-op-FFs}, we finally have the following theorem.
\begin{theorem} \label{isometric-base}
Let $\Sigma \hookrightarrow (M,g)$ be an analytic hypersurface embedding. Further, suppose that $\FF{2n} = 0$ for each $n$, that there exists a positive analytic function $f : \Sigma \hookrightarrow \mathbb{R}$ satisfying $\bar{\nabla}^2 f = f \III$, and that $\FF{2n+3} = O^f_n (\III)$ for each $n$. Then, in a neighborhood $U$ of a point in $\Sigma$, $(U,g|_U)$ is isometric to a warped product manifold with base manifold $U \cap \Sigma$.
\end{theorem}
\begin{proof}
Given the function $f$ and following Theorem~\ref{vanishing-evens} and Proposition~\ref{diff-op-FFs}, we may construct a base-like hypersurface embedding with induced metric and fundamental forms that match those given in the theorem. Hence it follows from Theorem~\ref{equal-metrics}, that, in a neighborhood of a point in $\Sigma$, we have that the given embedding restricts to a base-like hypersurface embedding.

\end{proof}

\section{Conformal Fundamental Forms and Products} \label{sec:conf}
In the same way that a Riemannian fundamental form is a natural Riemannian invariant, meaning that the quantity transforms covariantly with respect to coordinate transformations and can be constructed from natural geometric objects arising from the embedding (such as the metric, curvatures, and the unit normal vector), conformal fundamental forms are conformal invariants, meaning they transform covariantly under both coordinate transformations and conformal rescalings and are similarly constructible. Conformal fundamental forms, then, can be considered curvature corrections to the Riemannian fundamental forms such that they respect conformal rescalings. More generally, conformal fundamental forms are defined as below~\cite{BGW1}:
\begin{definition}
Let $2 \leq k \in \mathbb{N}$. A $k$th \textit{conformal fundamental form} is any natural section $\FFo{k}$ of $\odot^2_{\circ} T^* \Sigma$ with transverse order $k-1$ such that, for an embedding $\Sigma \hookrightarrow M$ with $M$ equipped with two metrics $g$ and $\Omega^2 g$, the section obey
$$\FFo{k}^{\Omega^2 g} = \Omega^{3-k} \FFo{k}^g\,.$$
\end{definition}

Observe that the trace-free second fundamental form $\IIo$ satisfies this definition, as does, for example, a third conformal fundamental form
$$\IIIo := W(n,\cdot, \cdot, n)|_{\Sigma}\,.$$
Conformal fundamental forms are examples of sections of weighted density bundles.

To elaborate on this fact, note that a conformal manifold $(M,\cc)$ can be identified with a ray subbundle $\mathcal{G} \in \odot^2 T^* M$ where the fiber at $p \in M$ is the set of all possible metrics $g \in \cc$ evaluated at $p$. Note that this subbundle is a principal bundle with structure group $\mathbb{R}_+$. A \textit{density bundle of weight} $w \in \mathbb{R}$ denoted by $\ce M[w]$, then, is a line bundle associated to $\mathcal{G}$ via the irreducible representation $\mathbb{R}_+ \ni t \mapsto t^{-w/2} \in \operatorname{End}(\mathbb{R})$. Note that this bundle can also be constructed as a real power of the volume forms
$$\ce M[w] \cong [(\wedge^d TM)^2]^{w/2d}\,,$$
and a section can be viewed as a double equivalence class $\phi = [g; f] = [\Omega^2 g; \Omega^w f] \in \Gamma(\ce M[w])$.
Then, for any vector bundle $\mathcal{V}$, we may define a weighted vector bundle $\mathcal{V}[w]$ as the tensor product $\mathcal{V}[w]:= \mathcal{V} \otimes \ce M[w]$. In that case,
$$\FFo{k} \in \Gamma(\odot^2_{\circ} T^* \Sigma[3-k])\,.$$

There are several constructions of conformal fundamental forms~\cite{BGW1,Blitz1}, and it is not known to what extent conformal fundamental forms can always be constructed. Indeed, it is known that when the conformal manifold $(M^d,\cc)$ has even dimension, conformal fundamental forms up to transverse order $d-2$ may be constructed for any conformal hypersurface embedding~\cite{Blitz1}. On the other hand, when the conformal manifold has odd dimension, conformal fundamental forms are constructible up to transverse order $\tfrac{d-1}{2}$~\cite{BGW1}. Going forward, when results involving those conformal fundamental forms $\FFo{k}$ with $k \geq 4$, we use the canonical constructions offered in~\cite{BGW1} when available or otherwise use the construction in~\cite{Blitz1}.

Several of the results in the section that follows assume that the dimension $d$ of the conformal manifold $(M^d,\cc)$ is even. However, all of these results may be reproduced in the odd dimensional case, except they only characterize or rely upon conformal fundamental forms up to transverse order $\tfrac{d-1}{2}$.

\medskip

Now we may study the relationship between conformal fundamental forms and the existence of a product metric $g \in \cc$ such that for some metric $\bar{g} \in \bar{\cc} := \iota^* \cc$, $g$ is isometric to $ds^2 + \bar{g}$, as we did for the relationship between Riemannian fundamental forms and warped product manifolds. We call such conformal hypersurface embeddings \textit{conformal product manifolds}. As we do not necessarily have access to a ``complete'' family of conformal fundamental forms (as perhaps one might hope for an infinite family while only a finite number are defined), we may also define an \textit{asymptotic conformal product manifold}, which is a conformal hypersurface embedding $\Sigma \hookrightarrow (M,\cc)$ for which there exists  $g \in \cc$ and $\bar{g} \in \bar{\cc}$ such that there exists an isometry $\Phi$
$$\Phi^* g = ds^2 +\bar{g} + \mathcal{O}(s^{d-1})\,.$$

\medskip

Having established the conformal equivalent of (warped) product manifolds, we begin with a conformally-invariant version of Theorem~\ref{vanishing-evens}.
\begin{lemma} \label{even-conf-FFs}
Let $\Sigma \hookrightarrow (M^d,\cc)$ be a conformal product manifold with $d$ even. For each integer $k \leq \tfrac{d-2}{2}$, $\FFo{2k} = 0$.
\end{lemma}
\begin{proof}
As $\Sigma \hookrightarrow (M,\cc)$ is a conformal product manifold, by definition, there exists a representative $g \in \cc$ and $\bar{g} \in \bar{\cc}$ such that 
there exists an isometry $\Phi$ such that  $\Phi^* g = ds^2 + \bar{g}$. We work in this metric representative.

Now note that all conformal fundamental forms are natural conformal hypersurface invariants, meaning they are expressible in terms of the metric $g$, the unit normal vector $n$, the covariant derivative $\nabla$, and the curvature tensor $R$~\cite{Blitz2}. Note that in any such formula, the only derivatives of $n$ that may appear are along the hypersurface, i.e. $(\nabla^\top)^m n$. Next, observe by direct computation of $\mathcal{L}_{n} g|_{\Sigma}$ that $\II = 0$. Thus, derivatives of $n$ may not contribute in a non-vanishing way to any conformal fundamental form. It follows that all natural conformal hypersurface invariants are expressible in terms of an undifferentiated $n$, the metric $g$, the curvature $R$, and its derivatives.

However, as the embedding is a product manifold, $\nabla = \partial_s + \bar{\nabla}$. Furthermore, $R^\top = \bar{R}$, $R_{abc n}^\top = 0$, $R_{nabn} = 0$, and $Ric = \bar{Ric}$, and there is no $s$ dependence in any curvature quantities. Thus, every natural conformal hypersurface invariant is expressible in terms of $\bar{g}$, $\bar{\nabla}$, and $\bar{R}$, i.e. these natural conformal hypersurface invariants are in fact expressible as natural conformal invariants of $(\Sigma, \bar{g})$. Note that $n$ cannot appear as natural conformal hypersurface invariants are tensors along $\Sigma$.

Now, observe that all pointwise-conformal invariants of $(\Sigma,\bar{g})$ (that do not involve the volume form) have even conformal weight, as $\bar{g}$ has homogeneity $2$, $\bar{\nabla}$ has homogeneity $0$, and $\bar{R}$ has homogeneity $0$ (with the natural index placement), where a Riemannian invariant $I$ with homogeneity $w$ is one that scales to $\lambda^w I$ when $g \mapsto \lambda^2 g$ for constant $\lambda$. But, $\FFo{2k} \in \Gamma(\odot^2_\circ T^* \Sigma[3-2k])$ has odd weight. As no composition of the intrinsic geometric quantities can produce an odd weight, there is only one section of this space that is built from geometric quantities: the zero section. This completes the proof.
\end{proof}

Like in the base-like hypersurface case for Riemannian warped product metrics, characterizing the odd conformal fundamental forms in the case of a conformal product manifold takes more finesse. As such, we first introduce some technical results.

The first of these technical results relates the third conformal fundamental form to the trace-free component of the Schouten tensor induced on the hypersurface.
\begin{proposition} \label{IIIo-Pb}
Let $\Sigma^{d-1} \hookrightarrow (\mathbb{R} \times \Sigma^{d-1},ds^2 + \bar{g})$. Then, $\IIIo = \tfrac{d-3}{d-2} \mathring{\bar{P}}$.
\end{proposition}
\begin{proof}
By definition, $\IIIo \= W_{nabn}$. First, observe from the relationship between the Weyl tensor and the Riemann curvature tensor, we have that
$$\otop R_{nabn} = W_{nabn} - \otop P_{ab}\,,$$
where $\otop$ is defined as the projection of a tensor the hypersurface (co)tangent bundle (or tensor products thereof) and subsequent projection to the trace-free part.
On a product manifold, however, it follows from Corollary~\ref{prod-metric} that $R_{nabn} = 0$, and so $W_{nabn} = \otop P_{ab}$. But by applying a standard hypersurface identity~\cite{Will1}
$$\IIo^2_{(ab)_{\circ}} - W_{nabn} = (d-3) \left(\otop P_{ab} - \mathring{\bar{P}}_{ab} + H \IIo_{ab} \right)\,,$$
and recalling from Proposition~\ref{even-conf-FFs} that $\IIo = 0$ for a product manifold, it follows that by direct computation that $\IIIo = \tfrac{d-3}{d-2} \mathring{\bar{P}}$. 
\end{proof}
The corollary then trivially follows:
\begin{corollary} \label{Einstein-IIIo}
Let $\Sigma \hookrightarrow (\mathbb{R} \times \Sigma,ds^2 + \bar{g})$. Then, $\IIIo = 0$ if and only if $(\Sigma,\bar{g})$ is Einstein.
\end{corollary}
%

Next, we are concerned with a solution to the singular Yamabe problem~\cite{Loewner,Aviles,ACF}, which is the problem of 
finding a defining density $\sigma \in \Gamma(\ce M[1])$ of $\Sigma = \mathcal{Z}(\sigma)$ such that 
$|I_{\sigma}|^2 := |d\sigma|^2 - \tfrac{2\sigma}{d}(\Delta \sigma + J \sigma) = 1$. 
(Note $\sigma$ is defining density if it is a defining function in any scale in $\cc$.)
It is known that there always exists such a density, and furthermore there always exists~\cite{Will1} a smooth solution to $|I_{\sigma}|^2 = 1 + \sigma^d \psi_d$, 
where $\psi_d \in \Gamma(\ce \Sigma[-d])$ is the celebrated Willmore invariant. 
We shall call the defining density $\sigma$ \idx{canonical} if it is such solution. Note
canonical defining density is given uniquelly modulo $\mathcal{O}(\tilde{\sigma}^{d+1})$ for any defining density 
$\tilde{\sigma}$.

\begin{theorem} \label{Yamabe-general}
Let $(\Sigma,\bar{g})$ has dimension greater than 2 and let 
$\Sigma \hookrightarrow (\mathbb{R} \times \Sigma,g:= ds^2 + \bar{g}_{ij} dx^i dx^j)$. Then, 
the canonical defining density that solves the singular Yamabe problem $\sigma$ is given an by
odd power series in $s$, i.e.
$$
\sigma = s + \sum_{\ell=1}^r \varphi_{2\ell+1} s^{2\ell+1}.
$$
More precisely, if $d$ is odd then $|I_\sigma|^2=1$ can be solved to all orders,
i.e.\ $r=\infty$, hence the Willmore invariant is zero.
If $d$ is even then the power series is deterermined only to the order $d-1$, i.e.\ for $r=\tfrac{d}{2}-1$.
Then $|I_\sigma|^2=1 + \psi_{d}s^d + O(s^{d+1})$ where $\psi_d$ is the Willmore invariant.
Moreover, all coefficients $\varphi_{2\ell+1}$  and also $\psi_{d}$
are polynomials in $J$ and its covariant derivatives.
\end{theorem}

The curvature quantity $J$ and the covariant derivative correspond to the metric $g$. But since
there is a very simple relation between covariant derivatives of $g$ and $\bar{g}$, the previous theorem
means coefficients $\varphi_{2\ell+1}$ and $\psi_d$ can be expressed 
as polynomials in $\bar{J}$ of $\bar{g}$ and its covariant derivatives with respect to  $\overline{\nabla}$.

\begin{proof}
We shall prove by induction (with respect to $\ell_0$) that we can choose
$\varphi_{2\ell+1}$, $0 \leq \ell \leq \ell_0 \leq r$ 
(as polynomials in $J$ and its covariant derivatives) in such a way that 
$|I_\sigma|^2 = 1 + \psi_{2(\ell_0+1)} s^{2(\ell_0+1)} + O(s^{2(\ell_0+1)+1})$. 
The case $\ell_0=0$---where the previous display is understood as
$\sigma=s$---can be easily  verified by a direct computation with $\psi_2 = -\tfrac{2J}{d}$.
Now assume we have 
$$
\sigma = s + \sum_{\ell=1}^{\ell_0} \varphi_{2\ell+1} s^{2\ell+1}
\quad \text{such that} \quad
|I_\sigma|^2 = 1 + \psi_{2(\ell_0+1)} s^{2(\ell_0+1)} + O(s^{2(\ell_0+1)+1})
$$
and put $\tilde{\sigma} = \sigma + \varphi_{2\ell_0+3} s^{2\ell_0+3}$ with $\varphi_{2\ell_0+3}$
to be determined. Using the notation $n_a := \nabla_a s$ (which satisfies $\nabla_a n_b =0$ and $|n|^2=1$), 
we compute
\begin{align*}
& \nabla_a \tilde{\sigma} = \nabla_a \sigma + (2\ell_0+3) \varphi_{2\ell_0+3} n_a s^{2\ell_0+2}
+ O(s^{2\ell_0+3}), \\
& \Delta \tilde{\sigma} = \Delta \sigma + (2\ell_0+2) (2\ell_0+3) \varphi_{2\ell_0+3} s^{2\ell_0+1}
+ O(s^{2\ell_0+2}).
\end{align*}
Since $\nabla_a \sigma = n_a + O(s^2)$, we further compute
\begin{align*}
& (\nabla^a \tilde{\sigma})(\nabla_a \tilde{\sigma}) = (\nabla_a \sigma)( \nabla_a \sigma) + 2(2\ell_0+3) \varphi_{2\ell_0+3} s^{2\ell_0+2} + O(s^{2\ell_0+3}), \\
& \tilde{\sigma} \Delta \tilde{\sigma} = \sigma \Delta \sigma + (2\ell_0+2) (2\ell_0+3) \varphi_{2\ell_0+3} s^{2\ell_0+2} + O(s^{2\ell_0+3}), \\
& \tilde{\sigma}^2 = \sigma^2 + O(s^{2\ell_0+4}).
\end{align*}
Summarizing, we obtain
\begin{equation} \label{Isigma}
\begin{split}
|I_{\tilde{\sigma}}|^2 &= |I_{\sigma}|^2 + \bigl[ 2(2\ell_0+3) - \tfrac{2}{d} (2\ell_0+2) (2\ell_0+3) \bigr]
 \varphi_{2\ell_0+3} s^{2\ell_0+2} + O(s^{2\ell_0+3}) = \\
& = 1 + \bigl[ \psi_{2(\ell_0+1)} + \tfrac{2(2\ell_0+3)(d-2\ell_0-2)}{d}  \varphi_{2\ell_0+3} \bigr]
s^{2\ell_0+2} + O(s^{2\ell_0+3})
\end{split}
\end{equation}
hence
$$
\varphi_{2\ell_0+3} = - \tfrac{d}{2(2\ell_0+3)(d-2\ell_0-2)} \psi_{2(\ell_0+1)}.
$$
Considering the denominator, this is always possible for $d$ odd. If $d$ is even, this is possible
for small $\ell$ up to $\ell = \tfrac{d}{2}-2$ but when $\ell = \tfrac{d}{2}-1$, the square bracket
in \eqref{Isigma} is independent on the choice of $\varphi_{2\ell_0+3}$; this bracket is the Willmore invariant
$\psi_{d}$~\cite{Will1}.

Finally note that if the square bracket in \eqref{Isigma} is zero, we have 
$|I_{\tilde{\sigma}}|^2 = 1 + O(s^{2\ell_0+3})$. However, a quick analysis of possible terms
in $|I_{\tilde{\sigma}}|^2$ reveals that this is necessarily an even power series and can only depend on $J$ and its derivatives. Hence
$|I_{\tilde{\sigma}}|^2 = 1 + O(s^{2(\ell_0+2)})$ which completes the proof.
\end{proof}

First a few coefficients of the power series for $\sigma$ can be computed by a direct combination. We obtain
\begin{align*}
\sigma =& s + \tfrac{1}{3(d-2)} Js^3 
+ \tfrac{1}{15(d-2)} \bigl[ \tfrac{1}{2(d-2)}J^2 + \tfrac{1}{d-4} \Delta J \bigr] s^5 + \\
& + \tfrac{1}{105(d-2)^2} \bigl[ \tfrac{1}{6(d-2)}J^3 - \tfrac{2(2d+3)}{3(d-4)(d-6)} J \Delta J 
- \tfrac{5d}{6(d-2)(d-6)} (\nabla^a J) (\nabla_a J) \bigr] s^7 + O(s^9).\\
\end{align*}
Similalry, we can evaluate the Willmore invariant $\psi_d$. We have $\psi_d=0$ for $d$ odd and further
$\psi_4 = -\tfrac{1}{12} \Delta J$ and 
$\psi_6 = \tfrac{1}{180} [J \Delta J + (\nabla^a J)(\nabla_a J) - \tfrac12 \Delta^2J]$.

In the special case of constant scalar curvature product manifolds, such $\sigma$ is constructible explicitly 
to all orders:
\begin{lemma} \label{Yamabe-soln}
Let $(\Sigma,\bar{g})$ be a constant scalar curvature Riemannian manifold with dimension greater than 2 and let $\Sigma \hookrightarrow (\mathbb{R} \times \Sigma,g:= ds^2 + \bar{g}_{ij} dx^i dx^j)$. Then, the canonical defining density that solves the singular Yamabe problem $\sigma$ is given by
\begin{align} \label{SY-const-sc}
\sigma := \left[g;  \begin{cases}
\sqrt{\tfrac{(d-1)(2-d)}{\bar{Sc}}} \sin\left(\sqrt{\tfrac{\bar{Sc}}{(d-1)(2-d)}} s\right) & \bar{Sc} < 0 \\
 s  & \bar{Sc}= 0 \\
 \sqrt{\tfrac{(d-1)(d-2)}{\bar{Sc}}} \sinh \left(\sqrt{\tfrac{\bar{Sc}}{(d-1)(d-2)}} s \right)  & \bar{Sc} > 0\,.
\end{cases} \right]
\end{align}
\end{lemma}
\begin{proof}
By the same arguments found in the proof of Lemma~\ref{even-conf-FFs}, we have that $(\mathbb{R} \times \Sigma,ds^2 + \bar{g})$ has $Sc = \bar{Sc}$. Now, we compute pointwise at $p \in \mathbb{R} \times \Sigma$.

Suppose first that $\bar{Sc}(x^i(p)) = 0$. Then, it trivially follows that $|ds|^2 - \tfrac{2}{d}(\Delta s + Js) = 1$, and so in the product metric representative, $\sigma = [g; s]$.

Next, suppose that $Sc < 0$. Then, suppose that $\sigma = [g; A \sin (rs)]$. The singular Yamabe equation becomes
$$A^2 r^2 \cos (rs)^2 + \frac{2A^2(r^2 - J)}{d} \sin(rs)^2 = 1\,.$$
Thus, we require that $A^2 r^2 = 1$ and that $2A^2 (r^2 - J)/d = 1$. Solving this system and substituting in $J = \tfrac{1}{2(d-1)} \bar{Sc}$, the result follows.

Finally, suppose that $Sc > 0$. Then, suppose that $\sigma = [g; A \sinh (rs)]$. The singular Yamabe equation becomes
$$A^2 r^2 \cosh (rs)^2 - \frac{2A^2(r^2 + J)}{d} \sinh(rs)^2 = 1\,.$$
Again, we require that $A^2 r^2 = 1$ and that $2A^2 (r^2 + J)/d = 1$. Solving again (for positive $A$ and $r$), the result for $\sigma$ follows.
\end{proof}

\begin{remark} \label{nonconst-Sc}
Theorem~\ref{Yamabe-general} and Lemma~\ref{Yamabe-soln} trivially imply that the canonical defining density for $\Sigma \hookrightarrow (\mathbb{R} \times \Sigma,g)$ with $(\Sigma, \bar{g})$ not having a constant scalar curvature differs only from the singular Yamabe density in Display~(\ref{SY-const-sc}) by terms involving derivatives of $\bar{Sc}$.
\end{remark}

%
From these technical results, we have the following technical theorem:
\begin{theorem} \label{E-PE}
Suppose that $(\Sigma, \bar{g})$ is Einstein and that $\Sigma \hookrightarrow (\mathbb{R} \times \Sigma,ds^2 + \bar{g})$. Then, $(\mathbb{R}_{\pm} \times \Sigma , g^+ := \frac{ds^2 + \bar{g}}{\sigma^2})$ is Poincar\'e--Einstein, where $\sigma$ is the singular Yamabe scale in the metric representative $ds^2 + \bar{g}$.
\end{theorem}
\begin{proof}
To prove this theorem, we rely on a result of Gover~\cite{Goal}, that if there exists a defining density $\sigma$ 
for the boundary $\Sigma$ such that
$$\nabla_{(a} \nabla_{b)\circ} \sigma + \sigma \mathring{P}_{ab} = 0\,,$$
i.e.\ $\sigma$ is an almost Einstein scale,
then the interior of the conformal manifold admits a Poincar\'e--Einstein metric representative in the scale $g^+$.

We first consider the case 
$\bar{Sc} < 0$, which is constant because $(\Sigma,\bar{g})$ is Einstein. We now calculate. First, observe that
$$\nabla \sigma = Ar \cos (rs) \, ds\,,$$
where $A$ and $r$ are as in Lemma~\ref{Yamabe-soln}. Next, we have that
$$\nabla^2 \sigma = -Ar^2 \sin (rs) ds \otimes ds\,.$$
Thus, the trace-free part of this tensor is
$$\nabla_{(a} \nabla_{b)\circ} \sigma = -Ar^2 \sin(rs) (n_a n_b - \tfrac{1}{d} g_{ab}) =  \tfrac{d-1}{d} Ar^2 \sin(rs) (\tfrac{1}{d-1} \bar{g}_{ab} - n_a n_b)$$

Now note that that
\begin{align*}
\mathring{P}_{ab} &= P_{ab} - \tfrac{1}{d} g_{ab} J \\
&= P^\top +n_a n_b P_{nn} - \tfrac{1}{d} {g}_{ab} J\\
&= \bar{P} + n_a n_b P_{nn} - \tfrac{1}{d} g_{ab} J  \\
&= \mathring{\bar{P}} + \tfrac{1}{d-1} \bar{g}_{ab} \bar{J} + n_a n_b P_{nn} - \tfrac{1}{d} g_{ab} J  \\
&=  \tfrac{1}{d-1} \bar{g}_{ab} \bar{J} + n_a n_b (J-\bar{J}) - \tfrac{1}{d} g_{ab} J \\
&= \tfrac{1}{d-1} \bar{g}_{ab} \bar{J} - \tfrac{1}{d-1} n_a n_b \bar{J} - \tfrac{d-2}{d(d-1)} g_{ab} \bar{J} \\
&= \tfrac{2}{d}\left(\tfrac{1}{d-1} \bar{g}_{ab} -  n_a n_b \right) \bar{J} \\
&= \tfrac{1}{d(d-2)}\left(\tfrac{1}{d-1} \bar{g}_{ab} -  n_a n_b \right) \bar{Sc}\,.
\end{align*}
Using that $r^2 = \frac{\bar{Sc}}{(d-1)(2-d)}$, one may easily check that
$$\nabla_{(a} \nabla_{b)\circ} \sigma + \sigma \mathring{P}_{ab} = 0\,,$$
as required.

When $\bar{Sc} = 0$, the identity trivially follows, as $\nabla^2 s = 0$. A similar calculation as the above follows for the case when $\bar{Sc} > 0$.
\end{proof}
As an aside, this theorem yields the following corollary:
\begin{corollary} \label{Einstein-vanishing-FFs2}
When $(\Sigma, \bar{g})$ is Einstein and $\Sigma \hookrightarrow (\mathbb{R}_{\pm} \times \Sigma, ds^2 + \bar{g})$ is a smooth embedding, $\FFo{k} = 0$ for each $k \in \{2, \ldots, d-1\}$, where $\FFo{k}$ is a canonical conformal fundamental form.
\end{corollary}
\begin{proof}
Because $(\Sigma,\bar{g})$ is Einstein, it follows that $(\mathbb{R}_{\pm} \times \Sigma,\tfrac{ds^2 + \bar{g}}{\sigma^2})$ is Poincar\'e--Einstein. But then the result follows trivially~\cite[Theorem 1.8]{BGW1}.
\end{proof}

Now, from the above technical results, we may derive the following relationship between conformal product manifolds and odd conformal fundamental forms.
\begin{proposition} \label{odd-conf-FFs}
Let $\Sigma^{d-1} \hookrightarrow (M^d := \mathbb{R} \times \Sigma, ds^2 + \bar{g})$ with $d \geq 4$ even. Then, $\IIIo = \tfrac{d-3}{d-2} \mathring{\bar{P}}$ and $\FFo{k} = \mathcal{O}_{k}(\IIIo)$ for each $k \in \{4 \ldots,d-1\}$, where $\mathcal{O}_{k}$ is some map on the jets of $\IIIo$ satisfying $\mathcal{O}_k(0) = 0$. Furthermore, $\IIo = 0 = \mathcal{O}_{2m}(\IIIo)$ for each $2 \leq m \leq \tfrac{d-2}{2}$.
\end{proposition}

\begin{proof}
First, observe from Proposition~\ref{IIIo-Pb}, it follows that $\IIIo = \tfrac{d-3}{d-2} \mathring{\bar{P}}$. So it suffices to show that all conformal fundamental forms are expressible in terms of a (potentially differential) operator on $\IIIo$ or equivalently $\mathring{\bar{P}}$.

Now, consider the case where $d = 4$. Then, from Lemma~\ref{even-conf-FFs} the proposition follows.

On the other hand, when $d \geq 6$, there  exists a conformally-invariant extension of the third conformal fundamental form~\cite{Blitz1}:
$$\IIIo^{\rm e}_{ab} := W_{\hat{n} ab \hat{n}} + 2 \sigma C_{\hat{n} (ab)} - \tfrac{\sigma^2}{d-4} B_{ab}\,,$$
where $\sigma$ is the singular Yamabe defining function described by Theorem~\ref{Yamabe-general} and $\hat{n} := d \sigma$. Now as the jets of $\IIIo^{\rm e}$ entirely capture the conformal fundamental forms from $\IIIo$ up to the $(d-1)$th conformal fundamental form~\cite{Blitz1}, to prove the proposition it is sufficient to show that $\IIIo^{\rm e}$ is expressible entirely in terms of $\IIIo$ and $\mathring{\bar{P}}$.

As $M$ is a product manifold, in the coordinates specified, $\nabla_n \mathcal{C} = 0$ for every tensor $C$ constructed from the metric and its derivatives. Furthermore, as $\sigma = \Omega s$, where
$$\Omega := \sum_{m = 0}^{\infty} \bar{\Omega}_{(m)} s^m\,,$$
with $\bar{\Omega}_{(m)}$ a coefficient that depends only on the induced scalar curvature and its derivatives (and $\bar{\Omega}_{(0)} = 1$), then a simple calculation yields
$$\nabla_n \sigma = 1 + \sum_{m=1}^{\infty} (m+1) \bar{\Omega}_{(m)} s^m$$
and
$$\nabla_n \hat{n}^a =  \sum_{m=1}^{\infty} \left[ (m+1) (\bar{\nabla}^a \bar{\Omega}_{(m)}) s + m(m+1) n^a \bar{\Omega}_{(m)} \right] s^{m-1} \,.$$
A consequence of this calculation is that, for $\ell \geq 1$,
$$\nabla_n^{\ell} \sigma \= \ell! \, \bar{\Omega}_{(\ell -1)}\,,$$
and
$$\nabla_n^{\ell} \hat{n}^a \= \ell !\, \bar{\nabla}^a \bar{\Omega}_{(\ell-1)} + (\ell + 1)! \, \bar{\Omega}_{(\ell)} n^a\,.$$

Now we may use these identities to evaluate the leading structure of the $(2k+3)$rd odd conformal fundamental form:
\begin{align*}
\otop \nabla_n^{2k} \IIIo^{\rm e}_{ab} &\=  \sum_{\ell +m = 2k} \Big [(\nabla_n^{\ell} \hat{n}^c)(\nabla_n^m \hat{n}^d ) W_{cabd} + 2 (\nabla_n^{\ell} \sigma)(\nabla_n^m \hat{n}^c) C_{c(ab)} - \tfrac{1}{d-4} (\nabla_n^{\ell} \sigma)(\nabla_n^m \sigma) B_{ab} \Big]^{\otop} \\
&\= \sum_{\ell +m = 2k} \Big [ \ell! m! (\bar{\nabla}^c \bar{\Omega}_{(\ell-1)})(\bar{\nabla}^d \bar{\Omega}_{(m-1)}) W_{cabd} + (\ell+1)!(m+1)! \bar{\Omega}_{(\ell)} \bar{\Omega}_{(m)} \IIIo_{ab} \\
 &\phantom{\= \sum_{\ell +m = 2k} \Big [}   + 2 \ell! m!\bar{\Omega}_{(\ell-1)} (\bar{\nabla}^c \bar{\Omega}_{(m-1)} ) \bar{C}_{c(ab)} + 2 \ell! (m+1)! \bar{\Omega}_{(\ell-1)} \bar{\Omega}_{(m)} C_{n(ab)}  \\
&\phantom{\= \sum_{\ell +m = 2k} \Big [} - \tfrac{1}{d-4} \ell! m! \bar{\Omega}_{(\ell-1)} \bar{\Omega}_{(m-1)} B_{ab} \Big]^{\otop}
\end{align*}
Because $\bar{\Omega}_{(m)}$ depends only on the induced scalar curvature and its derivatives (see the discussion after Theorem~\ref{Yamabe-general}) and
$$\bar{\nabla}_a \bar{Sc} = 2(d-1) \bar{\nabla}^b \mathring{\bar{P}}_{ab}\,,$$
it follows that all of the above terms containing $\bar{\nabla} \bar{\Omega}_{(m)}$ are expressible in terms of an operator on $\mathring{\bar{P}}$. Further recall that on product manifolds we have $\IIo = \IVo = 0$ according to 
Lemma~\ref{even-conf-FFs} and $H := \tfrac{1}{d-1} \II_a^a =  0$ according to Corrollary~\ref{prod-metric}. 
Thus it follows from the Codazzi-Mainardi equation \cite[(2.10)]{Will1} that $W_{abc\hat{n}}^\top=0$ hence the explicit
formula \cite[(1.3)]{BGW1} for $\IVo$ shows that $C_{n(ab)}^\top \= 0$. Finally, from an expression for the fifth conformal fundamental form in~\cite{BGW1} below the display (1.3), it follows that $B_{ab}^{\otop}$ is expressible in terms of the third fundamental form and the hypersurface Bach tensor $\bar{B}_{ab}$. But the hypersurface Bach tensor is expressible in terms of $\mathring{\bar{P}}$.

The above discussion shows that the leading transverse-order term in an odd-order conformal fundamental is expressible in terms of an operator on $\mathring{\bar{P}}$. By induction, it follows that the entirety of such an odd-order conformal fundamental form is expressible the same way, as the subleading terms are expressible in terms of lower-order conformal fundamental forms which are either even-order (which vanish) or odd-order, which are expressible in the same way.

\end{proof}

\medskip

Given Lemma~\ref{even-conf-FFs} and Proposition~\ref{odd-conf-FFs}, it is necessary that the even conformal fundamental forms vanish and the odd conformal fundamental forms take on particular forms in order for a conformal hypersurface 
 to admit a product metric in its conformal class. However, this is not sufficient. To prove this statement, we first need an identification theorem between conformal manifolds with matching conformal fundamental forms.

\begin{theorem} \label{conformal-identification}
Let $\iota : \Sigma \hookrightarrow (M^d,\cc)$ and $\hat{\iota} : \Sigma \hookrightarrow (M^d,\hat{\cc})$ with $d$ even and $\iota(p) = \hat{\iota}(p)$ for all $p \in \Sigma$. Furthermore, let $\sigma$ be a singular Yamabe density for $\iota$. Then, $\iota^* \cc = \hat{\iota}^* \hat{\cc}$ and $\FFo{k} = \hat{\FFo{k}}$ for each $k \leq d-1$ if and only if  there exists a diffeomorphism $\phi$ fixing $\Sigma$ pointwise such that
$$\phi^* \hat{\cc} = \cc + \mathcal{O}(\sigma^{d-1})\,.$$
\end{theorem}
\begin{proof}
To prove the forward direction, it suffices to show that, for a given choice of $\bar{g} \in \iota^* \cc$, there exists metrics $g \in \cc$, $\hat{g} \in \hat{\cc}$ and an isometry $\phi$ such that $\iota^* g = \hat{\iota}^* \hat{g} = \bar{g}$ and $g - \phi^* \hat{g} = \mathcal{O}(s^{d-1})$ where $\sigma = [g;s]$, i.e.\ $s$ is the representative of $\sigma$ in the scale $g$. 
So, picking $\bar{g}$, we let $g \in \cc$, $\hat{g} \in \hat{\cc}$ be the unique Graham--Lee metrics~\cite{GrahamLee} in the conformal classes, where for $\hat{\sigma} = [\hat{g}, \hat{s}]$ we have that $|ds|_g^2 = 1 =|d\hat{s}|_{\hat{g}}^2$. That is, we pick those metric representatives $g$ and $\hat{g}$ such that the defining functions $s$ and $\hat{s}$ determined by the singular Yamabe densities $\sigma$ and $\hat{\sigma} = [\hat{g}; \hat{s}]$ (which is the singular Yamabe density of $\hat{\iota}$) are the minimal geodesic distance functions between a point $p$ and $\Sigma$ as measured by the metric representatives $g$ and $\hat{g}$ respectively.

As $g$ and $\hat{g}$ are Graham--Lee metrics, it follows that there always exists isometries 
$\Phi,\hat{\Phi} : \Sigma \times I \rightarrow M$ such that, in a collar neighborhood of $\Sigma$,
\begin{align*}
\Phi^* g &= dt^2 + (\bar{g}_{ij}(\vec{x}) + t g^{(1)}_{ij}(\vec{x}) + \tfrac{1}{2} t^2 g^{(2)}_{ij}(\vec{x}) + \cdots) dx^i dx^j  \\
\hat{\Phi}^* \hat{g} &= dt^2 + (\bar{g}_{ij}(\vec{x}) + t \hat{g}^{(1)}_{ij}(\vec{x}) + \tfrac{1}{2} t^2 \hat{g}^{(2)}_{ij}(\vec{x}) + \cdots) dx^i dx^j \,.
\end{align*}
Here $I$ is an open interval containing $0$, i.e.\ $ \Sigma \times I \rightarrow M$ is a two-sided $\Sigma$ in $M$.
Note that $\Phi^*(s) = \hat{\Phi}^*(\hat{s}) = t$. Thus, the existence of the desired isometry $\phi$ is equivalent to $g^{(k)}(\vec{x}) = \hat{g}^{(k)}(\vec{x})$ for all $k \leq d-2$.

Observe that each coefficient $g^{(k)}$ is a natural hypersurface invariant (that is tangent to the hypersurface), as it can be obtained via $\mathcal{L}_n^{k} \Phi^* g|_{t = 0}$, where $n =\partial_t$, and similarly for $\hat{g}^{(k)}$. As such, from~\cite[Theorem 2.5]{Blitz2} it follows that $g^{(k)}$ is expressible as a partial contraction polynomial in
$$\{\bar{g}, \bar{g}^{-1}, \bar{\nabla}, \bar{R}, \IIo, \,\ldots\, \FFo{k+1}, H, J|_{\Sigma}, \ldots, \nabla_n^{k-2} J|_{\Sigma}\}\,,$$
and the expression is universal, meaning that the symbolic form of the expression does not depend on the underlying choice of metric.

Note that $|n|^2 = 1$, and so from the singular Yamabe equation $|n|^2 - \tfrac{2}{d} t(\Delta t + Jt) = 1 + \mathcal{O}(t^d)$, we have that $\Delta t + Jt = 0$ in a collar neighborhood of $\Sigma$. So it follows that $\nabla_n^k (\Delta t + Jt) = 0$ for every $1 \leq k \leq d-1$. Expanding this equation, we have
\begin{align*}
0 \=& \nabla_n^k \Delta t + (k-1) \nabla_n^{k-1} J \\
\=& \nabla_n^{k-1} [\nabla_n, \nabla_a] \nabla^a t + (k-1) \nabla_n^{k-1} J \\
\=& -\nabla_n^{k-1} Ric_{nn} + (k-1) \nabla_n^{k-1} J + \text{ltots} \\
\=& -(d-1) \nabla_n^{k-1} J + (k-1) \nabla_n^{k-1} J + \text{ltots} \\
\=& -(d-k) \nabla_n^{k-1} J + \text{ltots.}
\end{align*}
In the above, $\text{ltots}$ stands for ``lower transverse order terms,'' which are those terms which involve fewer normal derivatives of the metric. Note that the fourth identity follows from~\cite[Lemma 2.4]{Blitz1}. So, it is clear that for Graham-Lee metrics, $\nabla_n^{k-1} J|_{\Sigma}$ is expressible in terms of lower transverse order terms so long as $k < d$. But because $H = 0$ (because $H \= \tfrac{1}{d} \Delta t = -\rho|_{\Sigma}$), it follows by induction that \textit{in a Graham--Lee metric representative}, the coefficients $ g^{(k)}$ have universal expressions in terms of the set
$$\{\bar{g}, \bar{g}^{-1}, \bar{\nabla}, \bar{R}, \IIo, \,\ldots\, \FFo{k+1}\}\,.$$
As these expressions are universal (for Graham--Lee metric representatives), it follows that if the intrinsic geometry and the conformal fundamental forms agree for $\iota$ and $\hat{\iota}$, then viewed as tensors on $\Sigma$, $g^{(k)} =  \hat{g}^{(k)}$ for all $k \leq d-2$ (as the conformal fundamental forms are only defined up to that order).

\medskip

The reverse direction follows by noting that if any pair of conformal fundamental forms $\FFo{k}$, $\hat{\FFo{k}}$ are not equal, then it is clear from the Graham--Lee normal form of the metrics that they are not isometric.

\end{proof}

Using this theorem, we may now prove the following:
\begin{corollary} \label{classification-conformal-product}
Let $\Sigma \hookrightarrow (M^d,\cc)$ be a conformal hypersurface embedding with $d$ even
satisfying $\IIo_{ab}=0$. 
Suppose that there exists some $\tau \in \Gamma(\ce_+ \Sigma[1])$ such that
$$(\bar{\nabla}_{(a} \bar{\nabla}_{b)_{\circ}} + \mathring{\bar{P}}_{ab} - \tfrac{d-2}{d-3} \IIIo_{ab}) \tau = 0\,,$$
and that in the $\tau$ scale, $\FFo{k} = \mathcal{O}_k(\IIIo)$ with $3 \leq k \leq d-1$. Then $\Sigma \hookrightarrow (M,\cc)$ is an asymptotic conformal product manifold.
\end{corollary}
\begin{proof}
From Theorem~\ref{conformal-identification}, it suffices to show that the conformal hypersurface embedding $\iota: \Sigma \hookrightarrow (M,\cc)$ has conformal fundamental forms that agree with some asymptotic conformal product manifold. Now given a choice of representative $g_{\tau}$ determined by any extension of $\tau$ to $(M,\cc)$, we may identify an isometry $\Phi : \Sigma \times [0,\epsilon) \rightarrow M$ such that $\Phi^* g_{\tau}$ is in geodesic normal form. Now on the same collar neighborhood, we prescribe another metric $\hat{g} := dt^2 + \bar{g}_{ij}(\vec{x})dx^i dx^j + \mathcal{O}(t^{d-1})$, where $\Sigma = \mathcal{Z}(t)$ and $\bar{g} = \Phi^* \iota^* g_{\tau}$ with $\Phi^*$ acting in the natural way on $\Sigma$. This metric in turn defines a conformal class of metrics $\hat{\cc}$ in a collar neighborhood of $\Sigma$, with conformal embedding $\hat{\iota} : \Sigma \hookrightarrow (M,\hat{\cc})$, with $\hat{\iota}^* \hat{\cc} = \iota^* \cc$ and with conformal fundamental forms given by Proposition~\ref{odd-conf-FFs}: $\IIo = 0$, $\IIIo = \tfrac{d-3}{d-2} \mathring{\bar{P}}$, and $\FFo{k} = \mathcal{O}_k(\IIIo)$.

But in the scale $g_{\tau}$, we have that $\tau = [g_{\tau}; 1]$ and so by hypothesis, $\mathring{\bar{P}} = \tfrac{d-2}{d-3} \IIIo$ and $\FFo{k} = \mathcal{O}_k(\IIIo)$. Furthermore, by hypothesis $\IIo = 0$. As such, the conformal fundamental forms match the conformal fundamental forms of $\hat{\iota}: \Sigma \hookrightarrow (M,\cc)$. The corollary follows.

\end{proof}

\begin{remark}
The equation
$$(\bar{\nabla}_{(a} \bar{\nabla}_{b)_{\circ}} + \mathring{\bar{P}}_{ab} - \tfrac{d-2}{d-3} \IIIo_{ab}) \tau = 0\,,$$
in the statement of Corollary~\ref{classification-conformal-product} is a conformally invariant Jacobi-like differential equation for umbilic hypersurface embeddings $\Sigma \hookrightarrow (M^d,g)$. Solutions to this equation exist if and only if there exists an infinitesimal family of umbilic hypersurfaces generated by the vector field $\tau n$, where $n$ is the unit normal vector to $\Sigma$. This follows by demanding that the first variation of the trace-free second fundamental form with respect to the embedding vanishes.
\end{remark}

\bibliographystyle{siam}
\bibliography{products-RiemConf-MOFM}

\end{document}